\documentclass[12pt,leqno,fleqn]{amsart}  
\usepackage{amsmath,amstext,amsthm,amssymb,amsxtra}
\usepackage{txfonts} 
\usepackage[T1]{fontenc}
\usepackage{lmodern}

\usepackage{euler}   
\usepackage{amsfonts}
\usepackage{latexsym}

\theoremstyle{plain}
\newtheorem{thm}{Theorem}

\newtheorem{lem}[thm]{Lemma}
\newtheorem{prop}[thm]{Proposition}

\theoremstyle{definition}

\newtheorem*{rem*}{Remark}
\newtheorem*{ques*}{Question}

\newcommand{\N}{{\mathbb{N}}}

\newcommand{\R}{\mathbb{R}}

\newcommand{\Z}{\mathbb{Z}}

\DeclareMathOperator{\dist}{dist}
\DeclareMathOperator{\diam}{diam}

\numberwithin{equation}{section}

\usepackage{mathtools}
\mathtoolsset{showonlyrefs,showmanualtags} 
\def\kint_#1{\mathchoice%
          {\mathop{\kern 0.2em\vrule width 0.6em height 0.69678ex depth -0.58065ex
                  \kern -0.8em \intop}\nolimits_{\kern -0.4em#1}}%
          {\mathop{\kern 0.1em\vrule width 0.5em height 0.69678ex depth -0.60387ex
                  \kern -0.6em \intop}\nolimits_{#1}}%
          {\mathop{\kern 0.1em\vrule width 0.5em height 0.69678ex depth -0.60387ex
                  \kern -0.6em \intop}\nolimits_{#1}}%
          {\mathop{\kern 0.1em\vrule width 0.5em height 0.69678ex depth -0.60387ex
                  \kern -0.6em \intop}\nolimits_{#1}}}

\usepackage{hyperref} 
\hypersetup{
    colorlinks=true,       
    linkcolor=blue,          
    citecolor=magenta,        
    filecolor=magenta,      
    urlcolor=cyan           
}

\allowdisplaybreaks
\swapnumbers

\newtheorem*{ack}{Acknowledgment}
\numberwithin{equation}{subsubsection}

\title[Fractional Hardy inequality]{Characterizations for fractional Hardy inequality}


\author[B{.} Dyda]{Bart{\l}omiej Dyda}
\address[B.D.]{Faculty of Mathematics\\ University of Bielefeld\\
Postfach 10 01 31,
D-33501 Bielefeld, Germany\\
\and
 Institute of Mathematics and Computer Science\\ Wroc{\l}aw University of Technology\\
Wybrze\.ze Wyspia\'nskiego 27,
50-370 Wroc{\l}aw, Poland
}
\email{dyda@math.uni-bielefeld.de  bdyda@pwr.wroc.pl}

\author[A.V. V\"ah\"akangas]{Antti V. V\"ah\"akangas}
\address[A.V.V.]{Department of Mathematics and Statistics,
P.O. Box 68, FI-00014 University of Helsinki, Finland} \email{antti.vahakangas@helsinki.fi}

\begin{document}

\begin{abstract}
We provide a Maz'ya type characterization
for a fractional Hardy inequality.
As an application, 
we show that a bounded open set $G$ admits a fractional Hardy inequality if and only if
the associated fractional capacity is quasiadditive with respect to Whitney cubes of $G$
and the zero extension operator acting on $C_c(G)$ is bounded in an appropriate manner.
\end{abstract}

\maketitle

\section{Introduction}

An open set $\emptyset\not=G\subsetneq\R^n$ admits an $(s,p)$-Hardy
inequality, for $0<s<1$ and $0<p<\infty$, if there is a
constant $C>0$ such that inequality 
\begin{equation}\label{e.hardy}
\int_{G} \frac{\lvert u(x)\rvert^p}{\dist(x,\partial G)^{sp}}\,dx
\le  C\int_{G} \int_{G}
\frac{\lvert u(x)-u(y)\rvert ^p}{\lvert x-y\rvert ^{n+sp}}\,dy\,dx =: C\lvert u\rvert_{W^{s,p}(G)}^p
\end{equation}
holds for
every $u\in C_c(G)$.
Sufficient geometric conditions are available
for an open set to admit an $(s,p)$-Hardy inequality, e.g., in \cite{Dyda,Dyda2,ihnatsyeva3}.
However, these conditions are not necessary. Our first result, Theorem \ref{t.maz'ya},  is a Maz'ya-type characterization of the $(s,p)$-Hardy inequality in terms of
the following  $(s,p)$-capacities
of compact sets $K\subset G$: we write
\[
\mathrm{cap}_{s,p}(K,G) = \inf_u \lvert u\rvert_{W^{s,p}(G)}^p\,,
\]
where the infimum ranges over all real-valued $u\in C_c(G)$, i.e., continuous with
compact support in $G$, such that $u(x)\ge 1$ for $x\in K$.

\begin{thm}\label{t.maz'ya}
Let $0<s<1$ and $0<p<\infty$. An open set $\emptyset\not=G\subsetneq \R^n$ admits an $(s,p)$-Hardy inequality if and only if
there is a constant $c>0$ such that
\begin{equation}\label{e.Mazya}
\int_K \mathrm{dist}(x,\partial G)^{-sp}\,dx \le c\,\mathrm{cap}_{s,p}(K,G)
\end{equation}
for every compact set $K\subset G$.
\end{thm}

A Maz'ya-type characterization for  weighted embeddings gives Theorem \ref{t.maz'ya} as a special case,
see Proposition \ref{p.maz'ya} with $\omega(x)=\mathrm{dist}(x,\partial G)^{-sp}$.
In the proof of this proposition, we adapt the method that is used 
by Kinnunen and Korte to prove 
a Maz'ya-type characterization 
for the non-fractional Hardy inequality, \cite[Theorem 2.1]{MR2723821}.
This method, in turn, is based on a
 truncation argument
in the monograph of Maz'ya, \cite[p. 110]{MR817985}. For further information on this type of characterizations, we refer to \cite[\S 2]{MR817985} and \cite{MR2723821}.

A close connection between a non-fractional Hardy inequality and a `quasiadditivity property of 
the variational
capacity' was  recently found
by Lehrb\"ack and Shanmugalingam, \cite{L-S}. A Maz'ya-type characterization
has a significant role in their work, hence, it is not surprising that Theorem \ref{t.maz'ya} paves our way to 
the  analogous question, namely: {\em what is the connection between an $(s,p)$-Hardy inequality and quasiadditivity  of 
$\mathrm{cap}_{s,p}(\cdot,G)$
with respect to  a Whitney decomposition $\mathcal{W}(G)$ of an open set $G$?} In case of bounded open sets,
we  characterize the $(s,p)$-Hardy inequality in terms of  a (weak) quasiadditivity and
a  zero extension property, Theorem \ref{t.second}.
In order to state this partial extension of \cite[Corollary 3.5]{L-S} we need these definitions.

We say that the $(s,p)$-capacity $\mathrm{cap}_{s,p}(\cdot,G)$ is   (weakly) $\mathcal{W}(G)$-quasiadditive, if
there is a positive constant $N>0$ such that for every compact set $K\subset G$ 
(in the weak case for every $K=\cup_{Q\in\mathcal{E}} Q$, where $\mathcal{E}\subset\mathcal{W}(G)$ is finite),
\begin{equation}\label{e.quasi}
\sum_{Q\in\mathcal{W}(G)} \mathrm{cap}_{s,p}(K\cap Q,G) \le N\mathrm{cap}_{s,p}(K,G)\,.
\end{equation}
See \cite{L-S} for
information on the closely related quasiadditivity of variational capacity.

An open set $G$ is said to 
admit an $(s,p)$-zero extension, if there is a constant
$C>0$ such that the zero extension operator satisfies
\[\lvert E_G u\rvert_{W^{s,p}(\R^n)} \le C \lvert u\rvert_{W^{s,p}(G)}\]
for every function $u\in C_c(G)$. Here $E_G u(x)=u(x)$ if $x\in G$ and
$E_Gu(x)=0$ otherwise. 
Let us emphasise that 
only continuous functions with compact support need to have a bounded zero extension,
and not all open sets admit an $(s,p)$-zero extension.
We mention in passing that the usual extension
problem for $W^{s,p}(G)$ has been recently solved
by Zhou in \cite{Z}.

Observe that an open set admits an $(s,p)$-zero extension if, and only if, the
following weighted embedding holds, with a constant $c>0$ and a weight $\omega (x) = \int_{\R^n\setminus G} \lvert x-y\rvert^{-n-sp}\,dy$,
\begin{equation}\label{e.zero_ext}
\int_G \lvert u(x)\rvert^p \omega(x)\,dx \le c\int_G \int_G \frac{\lvert u(x)-u(y)\rvert^p}{\lvert x-y\rvert^{n+sp}}\,dy\,dx\,,\qquad u\in C_c(G)\,;
\end{equation}
Proposition \ref{p.maz'ya} provides a Maz'ya-type characterization for this weighted embedding, which
is 
 weaker than the $(s,p)$-Hardy inequality. Indeed, we have $\omega(x)\lesssim \dist(x,\partial G)^{-sp}$
if $x\in G$.

\begin{thm}\label{t.second}
Let $0<s<1$ and $1<p<\infty$ satisfy $sp<n$.
Suppose $G\not=\emptyset$ is a bounded open set in $\R^n$. Then the following conditions are equivalent.
\begin{itemize}
\item[(1)] $G$ admits an $(s,p)$-Hardy inequality;
\item[(2)] $\mathrm{cap}_{s,p}(\cdot,G)$ is $\mathcal{W}(G)$-quasiadditive and $G$ admits an $(s,p)$-zero extension;
\item[(3)] $\mathrm{cap}_{s,p}(\cdot,G)$ is weakly $\mathcal{W}(G)$-quasiadditive  and $G$ admits an $(s,p)$-zero extension.
\end{itemize}
Moreover, the implications (1) $\Rightarrow$ (2) $\Rightarrow$ (3) hold for unbounded  open sets $G\subsetneq \R^n$.
\end{thm}

Combined with sufficient conditions for an $(s,p)$-Hardy inequality,
Theorem \ref{t.second} yields sufficient conditions for the $\mathcal{W}(G)$-quasiadditivity of $\mathrm{cap}_{s,p}(\cdot,G)$. Another point-of-view is that
the two weighted embeddings \eqref{e.hardy} and \eqref{e.zero_ext} are
equivalent under  the $\mathcal{W}(G)$-quasiadditivity assumption.

The following question remains open to  our knowledge. It is  motivated by \cite{L-S}, where a positive answer is provided
in case of a non-fractional Hardy inequality.
Below $\ell(Q)$ stands for the side length of the cube $Q$.

\begin{ques*}
Is the condition
\begin{itemize}
\item[(4)] $\mathrm{cap}_{s,p}(\cdot,G)$ is weakly $\mathcal{W}(G)$-quasiadditive  and $\ell(Q)^{n-sp}\lesssim \mathrm{cap}_{s,p}(Q,G)$ if $Q\in\mathcal{W}(G)$
\end{itemize}
equivalent with condition (1) in Theorem \ref{t.second}? 
\end{ques*}

To state this otherwise, is it possible to 
replace the $(s,p)$-zero extension condition by a 
testing condition \eqref{e.Mazya} restricted to Whitney cubes $K=Q\in\mathcal{W}(G)$?
(The lower bound on the capacities of Whitney cubes may be viewed as such.)
The fact that $G$ need not  admit $(s,p)$-zero extension introduces complications to 
the treatment of this question, as the boundedness properties for the local maximal
operator $M_G$, established by Luiro in \cite{MR2579688}, are no longer available. And, our proof of Theorem \ref{t.second} relies on these properties instead
of, say, weak Harnack inequalities as in \cite{L-S}.

\medskip

The structure of this paper is the following. In \S \ref{s.notation} we present
notation and also properties of the local maximal operators.
The Maz'ya-type characterization, yielding Theorem \ref{t.maz'ya} in particular, is proven in \S \ref{s.maz'ya}, and the proof
of Theorem \ref{t.second} is divided in sections \S \ref{s.necessary} and \S \ref{s.sufficient}.
In \S \ref{s.counterexamples},
we provide two counterexamples showing that the two conditions occurring
in point (2) of Theorem \ref{t.second} are independent, i.e., neither one of them implies
the other one. In fact, the same examples show that the two conditions in either (3) or (4) are also independent.

\ack{
B.D. was supported in part by NCN grant 2012/07/B/ST1/03356
and by the DFG through SFB-701 'Spectral Structures and Topological Methods in Mathematics'.
The research was initiated
while the second author was visiting University of Bielefeld,
and he
would like to thank B. Dyda and M. Ka{\ss}mann for their hospitality.
The authors would like to thank Juha Lehrb\"ack and Heli Tuominen for useful discussions.
}

\section{Preliminaries}\label{s.notation}

\subsection{Whitney cubes} 
For an open set $\emptyset\not=G\subsetneq \R^n$, we fix its
Whitney decomposition $\mathcal{W}(G)$ consisting of closed  cubes
such that, for each $Q\in\mathcal{W}(G)$, 
\begin{equation}\label{dist_est}
  \diam(Q)\le \dist(Q,\partial G)\le 4\diam(Q)\,.
\end{equation}
We have  $\sum_{Q\in\mathcal{W}(G)} \chi_{Q^{**}} \le C_n \chi_G$, where
$Q^{**}=\tfrac 98 Q$,  and $G=\cup_{Q\in \mathcal{W}(G)}Q$, see \cite[VI.1]{MR0290095}.

\subsection{Function spaces}
Let us recall the definition of the fractional order Sobolev spaces in open sets $G\subset \R^n$.
For $0< p<\infty$ and $s\in (0,1)$ we let
$W^{s,p}(G)$ be the family of functions $u$ in
$L^p(G)$ with
\begin{equation}\label{frac_def}
\begin{split}
\lVert u \rVert_{W^{s,p}(G)}
&:=\lVert u\rVert_{L^p(G)} + \lvert u\rvert_{W^{s,p}(G)}\\&:= \bigg(\int_G \lvert u(x)\rvert^p\,dx\bigg)^{1/p}+ \bigg(\int_{G} \int_{G}
\frac{|u(x)-u(y)|^p}{|x-y|^{n+sp}}\,dy\,dx\bigg)^{1/p}<\infty\,.
\end{split}
\end{equation}
The global fractional Sobolev spaces belong to the well-known scale of Triebel--Lizorkin $F$-spaces:
if $1<p<\infty$ and $0<s<1$, then $W^{s,p}(\R^n)$ 
coincides with $F^{s}_{pp}(\R^n)$ and the associated norms are equivalent, 
\cite[pp. 6--7]{MR1163193}.
Let $G$ be an open set in $\R^n$, $1< p<\infty$,
and $0<s<1$. Then
\[
F^s_{pp}(G)=\{u\in L^p(G)\,:\, \text{there is a}\, g\in F^s_{pp}(\R^n)\,\,\text{with}\,g|_G=u\}
\]
\[
\Vert u\Vert_{F^s_{pp}(G)}=\inf\Vert g\Vert_{F^s_{pp}(\R^n)},
\]
where the infimum is taken over all $g\in F^s_{pp}(\R^n)$ such that $g|_G=u$ pointwise a.e.

\subsection{Local maximal operator} 
Let $\emptyset\not=G\subsetneq \R^n$ be an open set.
The  local Hardy--Littlewood maximal operator
$M_G$ is defined as follows. For a measurable $u:G\to \R$,
\[
M_G u(x) = \sup_r \fint_{B(x,r)} \lvert u(y)\rvert\,dy\,,\quad x\in G\,,
\]
where the supremum ranges over $0<r<\dist(x,\partial G)$.
The following statement is important to us: Luiro  has shown  that $M_G$ is bounded on $F^{s}_{pp}(G)$
if $1<p<\infty$ and $0<s<1$, we refer to \cite[Theorem 3.2]{MR2579688}.  This  yields the following lemma.

\begin{lem}\label{l.maximal}
Let $\emptyset\not= G\subsetneq \R^n$ be a bounded open set, which
admits an $(s,p)$-zero extension with $0<s<1$ and $1<p<\infty$. Then
\[
\lvert M_G u \lvert_{W^{s,p}(G)} \lesssim \lvert u\rvert_{W^{s,p}(G)}
\]
for every $u\in C_c(G)$.
\end{lem}

\begin{proof}
Without loss of generality, we may assume that $\lvert u\rvert_{W^{s,p}(G)}<\infty$.
We have
\begin{align*}
\lvert M_G u \lvert_{W^{s,p}(G)} &\lesssim \lVert M_G u \lVert_{F^{s}_{pp}(G)}
\\ &\lesssim \lVert u\rVert_{F^s_{pp}(G)}\le \lVert E_G u\rVert_{F^s_{pp}(\R^n)}
\lesssim \lVert E_G u\rVert_{L^p(\R^n)} + \lvert E_G u\rvert_{W^{s,p}(\R^n)}\,.
\end{align*}
Since $G$ admits an $(s,p)$-zero extension, the last seminorm is dominated by $C\lvert u\rvert_{W^{s,p}(G)}$.
Let us then fix a compact set $K\subset \R^n\setminus G$ for which $\lvert K\rvert >0$. By the boundedness
of $G$,
\begin{align*}
\lVert E_G u \rVert_{L^p(\R^n)}^p= \int_G \lvert u(x)\rvert^p\,dx 
\lesssim  \int_K \int_G \frac{\lvert E_Gu(x)- E_Gu(y) \rvert^p}{\lvert x-y\rvert^{n+sp}}\,dx\,dy \le \lvert E_G u\rvert_{W^{s,p}(\R^n)}^p\,.
\end{align*}
The last term is, again, dominated by $C\lvert u\rvert_{W^{s,p}(G)}^p$.
\end{proof}

For the convenience of the reader, we provide
the proof of the following useful lemma.

\begin{lem}\label{l.continuity}
Suppose $G$ is an open set in $\R^n$ and $u\in C_c(G)$. Then $M_G u$ is continuous.
\end{lem}

\begin{proof}
We first observe that the function
defined by $\overline{u}(x,r)=\fint_{B(x,r)} \lvert u(y)\rvert\,dy$ for $r>0$
and $\overline{u}(x,0)=\lvert u(x)\rvert$ is continuous on $\R^n\times [0,\infty)$
(in this definition the function $\lvert u\rvert$ is zero-extended to the
whole $\R^n$). Let us fix $x\in G$ and $\varepsilon>0$.
By uniform continuity of the function $\overline{u}$
on $\overline{B}(x,\dist(x,\partial G)) \times [0,2\dist(x,\partial G)]$,
there exists $0<\delta<\dist(x,\partial G)$ such that
\[
 |\overline{u}(y,s)- \overline{u}(x,t)| < \varepsilon,
\]
whenever $|y-x|+|s-t|<\delta$ and $0\leq s,t \leq 2\dist(x,\partial G)$.
Therefore, if $y\in B(x,\delta/2)$, then for some $r_0=r_0(y,\varepsilon) < \dist(y,\partial G)$
\[
 M_G u(y) \leq \overline{u}(y,r_0) + \varepsilon \leq
   \overline{u}(x,r_0 \wedge \dist(x,\partial G)) + 2\varepsilon \leq M_Gu(x)+2\varepsilon,
\]
because $|y-x| + |r_0 - r_0 \wedge \dist(x,\partial G)| < \delta$.
On the other hand, for some $r_0\in [0,\dist(x,\partial G))$,
\[
 M_G u(x) \leq \overline{u}(x,r_0) +  \varepsilon \leq
 \overline{u}(y,r_0 \wedge \dist(y,\partial G)) + 2\varepsilon
\leq  M_G u(y)+ 2\varepsilon.
\]
This proves continuity of $M_G u$.
\end{proof}

\section{Maz'ya-type characterization}\label{s.maz'ya}

Theorem \ref{t.maz'ya} is implied by the following Maz'ya-type characterization for weighted embeddings, 
applied with  $\omega(x)=\dist(x,\partial G)^{-sp}$.
Let us also remark that inequality \eqref{e.zero_ext} admits a similar characterization
with a weight $\omega(x)=\int_{\R^n\setminus G} \lvert x-y\rvert^{-n-sp}\,dy$.

\begin{prop}\label{p.maz'ya}
Suppose $G\subset\R^n$ is an open set and $\omega:G\to [0,\infty)$ is a measurable function. The
following two conditions are equivalent for $0<s<1$ and $0<p<\infty$.
\begin{itemize}
\item[(A)] There is a constant $C>0$ such that 
\[\int_G \lvert u(x)\rvert^p\omega(x)\,dx \le C \lvert u\rvert_{W^{s,p}(G)}^p\,,\qquad u\in C_c(G)\,.\]
\item[(B)] There is a constant $c>0$ such that, for every compact set $K\subset G$,
\[
\int_K \omega(x)\,dx \le c\, \mathrm{cap}_{s,p}(K,G)\,.
\]
\end{itemize}
\end{prop}

\begin{proof}
First assume condition (A) holds.
Let $u\in C_c(G)$ be such that $u(x)\geq 1$ for every $x\in K$.
By (A) we obtain
\[
 \int_K \omega(x)\,dx \leq \int_G \lvert u(x)\rvert^p \omega(x)\,dx
\leq C \int_G
 \int_G \frac{\lvert u(x)-u(y)\rvert^p}{\lvert x-y\rvert^{n+sp}}\, dy\, dx\,.
\]
Taking infimum over all such functions $u$, we obtain (B) with $c=C$.

Now assume that (B) holds and let $u\in C_c(G)$. For $k\in \Z$
denote
\[
 E_k = \{x\in G : |u(x)| > 2^k \} \quad \text{and} \quad A_k = E_k \setminus E_{k+1}\,.
\]
Observe that
\begin{equation}\label{e.decomp}
G= \{x\in G\,:\, 0\le u(x)<\infty\} = \underbrace{\{x\in G \,:\, u(x)=0\}}_{=:F}\cup \bigcup_{i\in \Z} A_i\,.
\end{equation}
Hence, by (B) we obtain
\begin{equation}\label{e.distcap}
\begin{split}
 \int_G \lvert u(x)\rvert^p \omega(x)\,dx
 &\leq
 \sum_{k\in \Z} 2^{(k+2)p} \int_{A_{k+1}}  \omega(x)\,dx 
\\&\leq
 c 2^{2p} \sum_{k\in \Z} 2^{kp} \mathrm{cap}_{s,p}(\overline{A}_{k+1}, G)  \,.
\end{split}
\end{equation}
Define $u_k:G\to [0,1]$ by
\[
u_k(x) = \begin{cases}
1, & \text{if $|u(x)| \geq 2^{k+1}$,}\\
\frac{|u(x)|}{2^k}-1 &\text{if $2^k < |u(x)| < 2^{k+1}$,}\\
0, & \text{if $|u(x)| \leq 2^k$.}
\end{cases}
\]
Then $u_k \in C_c(G)$ and it satisfies $u_k=1$ on $\overline{E}_{k+1} \supset \overline{A}_{k+1}$,
hence we may take it as a test function for the capacity. By recalling also \eqref{e.decomp}, we obtain that
\begin{equation}\label{e.plug}
\begin{split}
\mathrm{cap}_{s,p}(\overline{A}_{k+1}, G) &\leq 
\int_G \int_G \frac{|u_k(x)-u_k(y)|^p}{|x-y|^{n+sp}}\,dy\,dx\\
&\leq
 2 \bigg\{  \sum_{i\leq k} \sum_{j\geq k} \int_{A_i} \int_{A_j} + \sum_{j\ge k} \int_{F} \int_{A_j} \bigg\}\frac{|u_k(x)-u_k(y)|^p}{|x-y|^{n+sp}}\,dy\,dx\,.
\end{split}
\end{equation}
We observe that $|u_k(x)-u_k(y)| \leq 2^{-k} |u(x)-u(y)|$. Moreover, if $x\in A_i$ and $y\in A_j$,
where $i+2\leq j$, then
$ |u(x)-u(y)| \geq |u(y)| - |u(x)| \geq 2^{j-1}$,
hence $|u_k(x) - u_k(y)| \leq 1 \leq 2\cdot 2^{-j} |u(x)-u(y)|$.
Therefore,
\[
 |u_k(x)-u_k(y)| \leq 2\cdot 2^{-j} |u(x)-u(y)|, \quad (x,y)\in  A_i\times A_j\,,
\]
whenever $i\leq k \leq j$. Thus, 
\begin{align*}
  \sum_{i\leq k} \sum_{j\geq k} & \int_{A_i} \int_{A_j} \frac{|u_k(x)-u_k(y)|^p}{|x-y|^{n+sp}}\,dy\,dx\\
& \le  2^{p} \sum_{i\leq k} \sum_{j\geq k} 2^{-jp} \int_{A_i} \int_{A_j} \frac{|u(x)-u(y)|^p}{|x-y|^{n+sp}}\,dy\,dx\,.
\end{align*}
A similar argument show that
\begin{align*}
\sum_{j\ge k} \int_{F} \int_{A_j} \frac{|u_k(x)-u_k(y)|^p}{|x-y|^{n+sp}}\,dy\,dx
 \le  2^{p} \sum_{j\ge k}  2^{-jp} \int_{F} \int_{A_j} \frac{|u(x)-u(y)|^p}{|x-y|^{n+sp}}\,dy\,dx\,.
\end{align*}
Since $\sum_{k=i}^j 2^{(k-j)p} < \sum_{k=-\infty}^j 2^{(k-j)p}  \le \frac{1}{1-2^{-p}}$, we may
apply inequalities \eqref{e.distcap} and \eqref{e.plug} and then change the order of summations to obtain that
\[
\int_G \lvert u(x)\rvert^p\omega(x)\,dx
 \leq
 \frac{c 2^{3p+2}}{1-2^{-p}} 
\int_G \int_G \frac{\lvert u(x)-u(y)\rvert^p}{\lvert x-y\rvert^{n+sp}}\, dy\, dx\,.
\]
Thus condition (A) is satisfied with $C = \frac{c 2^{3p+2}}{1-2^{-p}}$.
\end{proof}

\section{Necessary conditions for Hardy}\label{s.necessary}

In this section, we prove the implication (1) $\Rightarrow$ (2) in Theorem \ref{t.second}.

\begin{prop}\label{p.necessary}
Suppose $G$ is an open set in $\R^n$ which admits an $(s,p)$-Hardy inequality with $0<s<1$ and $0< p<\infty$. Then $\mathrm{cap}_{s,p}(\cdot,G)$
is $\mathcal{W}(G)$-quasiadditive 
and $G$ admits an $(s,p)$-zero extension.
\end{prop}

\begin{proof}
The conclusion that $G$ admits an $(s,p)$-zero extension follows from
\cite[Lemma 6.1]{ihnatsyeva3}. 

Fix a compact set $K\subset G$. We still need to show that
\begin{equation}\label{e.wanted}
\sum_{Q\in\mathcal{W}(G)} \mathrm{cap}_{s,p}(K\cap Q,G)
\le N \inf \int_G \int_G \frac{\lvert u(x)-u(y)\rvert^{p}}{\lvert x-y\rvert^{n+sp}}\,dy\,dx\,,
\end{equation}
where the infimum is taken over all $u\in C_c(G)$ such that $u\ge 1$ on $K$.
Fix a function $u$ for which the infimum is essentially obtained, say, within
a factor of $2$. 

For each Whitney cube $Q\in\mathcal{W}(G)$ we 
fix
a smooth function $\phi_Q$ such that
$\chi_{Q}\le \phi_Q\le \chi_{Q^*}$, where $Q^*=\tfrac {17}{16}Q$, and $\lvert \nabla \phi_Q\rvert \lesssim \ell(Q)^{-1}$. In particular, we find that $u_Q:= u\phi_Q$ is a test function for $\mathrm{cap}_{s,p}(K\cap Q,G)$. Hence, with $Q^{**}= \tfrac{9}{8}Q$,
\begin{align*}
LHS\eqref{e.wanted} &\le \sum_{Q\in\mathcal{W}(G)}
\int_G \int_G \frac{\lvert u_Q(x)-u_Q(y)\rvert^{p}}{\lvert x-y\rvert^{n+sp}}\,dy\,dx\\
&\lesssim \sum_{Q\in\mathcal{W}(G)}\bigg\{    
\int_{Q^*} \frac{\lvert u_Q(x)\rvert^p}{\mathrm{dist}(x,\partial G)^{sp}}\,dx+
\int_{Q^{**}} \int_{Q^{**}}\frac{\lvert u_Q(x)-u_Q(y)\rvert^{p}}{\lvert x-y\rvert^{n+sp}}\,dy\,dx
 \bigg\}
\end{align*}
Since $\lvert u_Q\rvert \le \lvert u\rvert$ and $\sum_{Q} \chi_{Q^*}\lesssim 1$, we may
apply the $(s,p)$-Hardy inequality for
\begin{align*}
\sum_{Q\in\mathcal{W}(G)}\int_{Q^*} \frac{\lvert u_Q(x)\rvert^p}{\mathrm{dist}(x,\partial G)^{sp}}\,dx
\lesssim \int_G \frac{\lvert u(x)\rvert^p}{\mathrm{dist}(x,\partial G)^{sp}}\,dx\lesssim \mathrm{cap}_{s,p}(K,G)\,.
\end{align*}
To estimate the second series, we first find that
\begin{align*}
\lvert u_Q(x)-u_Q(y)\rvert &= \lvert u(x)\phi_Q(x) - u(x)\phi_Q(y) + u(x)\phi_Q(y) - u(y)\phi_Q(y)\rvert\\
&\le  \lvert u(x)\rvert \lvert \phi_Q(x)-\phi_Q(y)\rvert + \lvert u(x)-u(y)\rvert \phi_Q(y)\\
&\lesssim \lvert u(x)\rvert \lvert x-y\rvert \ell(Q)^{-1} + \lvert u(x)-u(y)\rvert\,.
\end{align*}
Since $\sum_{Q\in\mathcal{W}(G)} \chi_{Q^{**}}\lesssim \chi_G$, we find that
\begin{align*}
 \sum_{Q\in\mathcal{W}(G)}
\int_{Q^{**}} \int_{Q^{**}}\frac{\lvert u(x)-u(y)\rvert^{p}}{\lvert x-y\rvert^{n+sp}}\,dy\,dx
\lesssim \mathrm{cap}_{s,p}(K,G)\,.
\end{align*}
To estimate the remaining series, we proceed as follows (recall that $0<s<1$);
\begin{align*}
&\sum_{Q\in\mathcal{W}(G)}
\ell(Q)^{-p}\int_{Q^{**}} \lvert u(x)\rvert ^p \int_{Q^{**}}\frac{\lvert x-y\rvert^p}{\lvert x-y\rvert^{n+sp}}\,dy\,dx\\
&\lesssim \sum_{Q\in\mathcal{W}(G)}
\ell(Q)^{-sp}\int_{Q^{**}} \lvert u(x)\rvert ^p \,dx\\
&\lesssim \sum_{Q\in\mathcal{W}(G)}
\int_{Q^{**}} \frac{\lvert u(x)\rvert ^p}{\mathrm{dist}(x,\partial G)^{sp}}\, dx\lesssim \mathrm{cap}_{s,p}(K,G)\,,
\end{align*}
and this concludes the proof.
\end{proof}

%
%
%
%

\section{Sufficient conditions for Hardy}\label{s.sufficient}

The implication (3) $\Rightarrow$ (1) in Theorem \ref{t.second} is established here.
Let us observe that the remaining implication (2) $\Rightarrow$ (3) is trivial.
The outline of the proof below is from \cite{L-S}.

\begin{prop}\label{p.sufficient}
Let $0<s<1$ and $1<p<\infty$ satisfy $sp<n$.
Suppose that $G\subset \R^n$ is a bounded open set such that  $\mathrm{cap}_{s,p}(\cdot,G)$ is weakly $\mathcal{W}(G)$-quasiadditive
and $G$ admits an $(s,p)$-zero extension. 
Then $G$ admits an $(s,p)$-Hardy inequality.
\end{prop}

\begin{proof} 
By Theorem \ref{t.maz'ya}, it suffices to show that
\begin{equation}\label{e.wish}
\int_K \mathrm{dist}(x,\partial G)^{-sp}\, dx \lesssim \mathrm{cap}_{s,p}(K,G)\,,
\end{equation}
where $K\subset G$ is compact.
We fix a test function $u$ for $\mathrm{cap}_{s,p}(K,G)$ such that the infimum in the definition of $(s,p)$-capacity is
obtained within a factor  $2$. By replacing $u$ with $\max\{0,\min \{u, 1\}\}$ we may
assume that $0\le u\le 1$. The truncation
can be written as $f\circ u$, where $f$ is $1$-Lipschitz, hence
this truncation does not increase the associated seminorm.

Let us  split $\mathcal{W}(G)= \mathcal{W}_1\cup \mathcal{W}_2$, where 
\[
\mathcal{W}_1 = \{Q\in\mathcal{W}(G)\,:\, \langle u\rangle_Q :=\fint_Q u < 1/2\}\,,\qquad \mathcal{W}_2 = 
\mathcal{W}(G)\setminus \mathcal{W}_1\,.
\]
Write the left-hand side of \eqref{e.wish} as
\begin{equation}\label{e.w_split}
 \bigg\{\sum_{Q\in\mathcal{W}_1} + \sum_{Q\in\mathcal{W}_2}\bigg\}
\int_{K\cap Q} \mathrm{dist}(x,\partial G)^{-sp}\,dx\,.
\end{equation}
To estimate the first series we observe that, for $x\in K\cap Q$ with $Q\in\mathcal{W}_1$,
\[
\tfrac 12=1-\tfrac12 <  u(x) -  \langle u\rangle_Q  = \lvert u(x)-\langle u\rangle_Q\rvert\,.
\]
Thus, by Jensen's  inequality,
\begin{align*}
\sum_{Q\in\mathcal{W}_1} 
\int_{K\cap Q} \mathrm{dist}(x,\partial G)^{-sp}\,dx
&\lesssim \sum_{Q\in\mathcal{W}_1} 
\ell(Q)^{-sp} \int_{Q} \lvert u(x)-\langle u\rangle_Q\rvert^p\,dx\\
&\lesssim \sum_{Q\in\mathcal{W}_1} 
\ell(Q)^{-n-sp} \int_{Q}\int_{Q} \lvert u(x)-u(y)\rvert^p\,dy\,dx\\
&\lesssim \sum_{Q\in\mathcal{W}_1} 
\int_{Q}\int_{Q} \frac{\lvert u(x)-u(y)\rvert^p}{\lvert x-y\rvert^{n+sp}}\,dy\,dx
\lesssim \mathrm{cap}_{s,p}(K,G)\,.
\end{align*}

Let us then focus on the remaining series in \eqref{e.w_split}, namely, the one over $\mathcal{W}_2$.
We first establish two auxiliary estimates \eqref{e.m_est} and \eqref{e.cap_est}.

By inequality \eqref{dist_est}, 
for every $Q\in\mathcal{W}_2$ and $x\in \mathrm{int}(Q)$,
\begin{equation}\label{e.m_est}
\begin{split}
M_G u(x) 
\gtrsim \fint_Q u(y)\,dy \geq \tfrac 12.
\end{split}
\end{equation}
The support of $M_G u$ is a compact set in $G$ due to the boundedness of $G$
and the fact that $u\in C_c(G)$. 
By Lemma \ref{l.continuity}, we find that $M_G u$ is continuous. Concluding from these facts, we find that 
there is  $\rho>0$, depending only on $n$, such that $\rho M_Gu$ is an admissible test function for 
 $\mathrm{cap}_{s,p}(\cup_{Q\in\mathcal{W}_2} Q,G)$.

 Another useful estimate for $Q\in\mathcal{W}_2$ is a lower bound for its capacity, namely,
 \begin{equation}\label{e.cap_est}
 \ell(Q)^{n-sp}\lesssim \mathrm{cap}_{s,p}(Q,G)\,.
 \end{equation}
To verify this inequality, let $u_Q\in C_c(G)$ be a test function for $\mathrm{cap}_{s,p}(Q,G)$.
A fractional Sobolev embedding theorem, \cite[Theorem 6.5]{Nezza} with $p^*=np/(n-sp)$, 
and the assumption that $G$ admits an $(s,p)$-zero extension
yield
\[
\ell(Q)^{n-sp}\le \lVert E_G u_Q\rVert_{L^{p*}(\R^n)}^p \lesssim \lvert E_G u_Q \rvert_{W^{s,p}(\R^n)}^p \lesssim \lvert u_Q\rvert_{W^{s,p}(G)}^p\,.
\]
It remains to infimize the right hand side over functions $u_Q$.

We may continue as follows.
By 
\eqref{e.m_est}, \eqref{e.cap_est}, and the assumed weak quasiadditivity
with a finite union $K=\cup_{Q\in\mathcal{W}_2} Q$, we find that
\begin{align*}
\sum_{Q\in\mathcal{W}_2} 
\int_{K\cap Q} \mathrm{dist}(x,\partial G)^{-sp}\,dx 
&\lesssim \sum_{Q\in\mathcal{W}_2} 
\ell(Q)^{n-sp}\\
&\lesssim \sum_{Q\in\mathcal{W}_2} \mathrm{cap}_{s,p}(Q,G) \\
&\le N \mathrm{cap}_{s,p}\Big(\bigcup_{Q\in\mathcal{W}_2} Q,G\Big)\\
&\le N\rho^p \int_G \int_G \frac{\lvert M_Gu(x)-M_Gu(y)\rvert^p}{\lvert x-y\rvert^{n+sp}}\,dy\,dx\,.
\end{align*}
By Lemma \ref{l.maximal},  the last term is dominated by 
\[
CN\rho^p\lvert u\rvert_{W^{s,p}(G)}^p\lesssim \mathrm{cap}_{s,p}(K,G)\,,
\]
and this concludes the proof.
\end{proof}

\section{Counterexamples}\label{s.counterexamples}

We provide two domains as counterexamples, showing that
neither one of the two conditions in point (2) of Theorem \ref{t.second} is implied by the other one.
In fact, the counterexamples show that the same is true for points (3) and (4) stated
in the Introduction.

 In both of the
constructions here, we rely on computations in \cite{Dyda}.

\begin{thm}
The cube $G=(0,1)^n\subset \R^n$
does not admit $(s,p)$-zero extension and $\mathrm{cap}_{s,p}(\cdot,G)$
is $\mathcal{W}(G)$-quasiadditive  if $sp<1$, where $0<s<1$ and $1<p<\infty$.
\end{thm}

\begin{proof}
By computations in \cite[\S 2]{Dyda}, we find that $\mathrm{cap}_{s,p}(K,G)=0$ for
every compact set $K\subset G$. Therefore $\mathrm{cap}_{s,p}(\cdot,G)$ is
trivially $\mathcal{W}(G)$-quasiadditive. It is also shown by Dyda
that $G$ does not admit an $(s,p)$-Hardy inequality. Hence, by Proposition \ref{p.sufficient},
we find that $G$ does not admit an $(s,p)$-zero extension.
Alternatively, we may observe that condition (B) in Proposition \ref{p.maz'ya} fails
for the weight $\omega(x)=\int_{\R^n\setminus G} \lvert x-y\rvert^{-n-sp}\,dy$ in \eqref{e.zero_ext}.
\end{proof}

\begin{thm}
Let $0<s<1$ and $1<p<\infty$ satisfy $sp=1$.
There is a bounded domain $G\subset \R^2$ which admits an $(s,p)$-zero extension and
$\mathrm{cap}_{s,p}(\cdot,G)$ is not weakly $\mathcal{W}(G)$-quasiadditive.
\end{thm}

\begin{proof}
Let $G'$ be the standard Koch snowflake domain in $\R^2$, and
fix a closed cube $R\subset G'$. Define $G=G'\setminus L$, where
$L=x_R + [-\ell(R)/4,\ell(R)/4]\times \{0\}\subset R$ and
$x_R$ is the midpoint of $R$.
The domain $G'$ admits an $(s,p)$-Hardy inequality by \cite[Theorem 2]{Dyda2},
and therefore $G'$ admits an $(s,p)$-zero extension, see e.g. Proposition \ref{p.necessary}.
Since $G\subset G'$ and $\lvert G'\setminus G\rvert = \lvert L\rvert = 0$,
we see that also $G$ admits
an $(s,p)$-zero extension, see the proof of \cite[Theorem 6.5]{ihnatsyeva3}.

It remains to  verify that $\mathrm{cap}_{s,p}(\cdot,G)$ is not
weakly $\mathcal{W}(G)$-quasiadditive.

 Reasoning as in the
proof of inequality \eqref{e.cap_est} with  $sp=1<2=n$, we find that
\[
\ell(Q)^{n-sp}\lesssim \mathrm{cap}_{s,p}(Q,G)\,,\qquad Q\in\mathcal{W}(G)\,.
\]
For $m\in \N$, we let $\mathcal{W}^m$ denote the family of Whitney cubes $Q\in\mathcal{W}(G)$
that are contained in $R$ and satisfy  $\diam(Q)\ge 1/(2m)$. 
Let $K_m\subset G$ be the union of cubes in $\mathcal{W}^m$. Then, for each $m$,
\begin{align*}
\int_{K_m} \dist(x,L)^{-1}\,dx &\le  \int_{K_m} \dist(x,\partial G)^{-1}\,dx \\&\lesssim
\sum_{Q\in\mathcal{W}^m} \ell(Q)^{n-sp}
\\&\lesssim \sum_{Q\in\mathcal{W}^m} \mathrm{cap}_{s,p}(Q,G)\le
\sum_{Q\in\mathcal{W}(G)} \mathrm{cap}_{s,p}(K_m\cap Q,G)\,.
\end{align*}
Observe
that $K_1\subset K_2\subset \dotsb$
and $(L+B(0,\epsilon))\setminus L\subset \cup_{m=1}^\infty K_m$ for an appropriate $\epsilon>0$.
Thus, the left hand side tends to $\infty$, as $m\to \infty$. 
In particular, we may infer that the last terms, as a function of $m$, also tends to $\infty$, as $m\to\infty$.

We have  $K_m\subset R\setminus ( L + B(0,1/(2m)))=:R\setminus L_m$. Indeed, for every $Q\in\mathcal{W}^m$,
\[
\frac{1}{2m} \le \diam(Q)\le \dist(Q,\partial G)\le \dist(Q,L)\,.
\]
Thus, in order to finish the proof, it suffices
to find functions $u_m\in C_c(G)$ satisfying $u_m\ge 1$ on $R\setminus L_m$ and
$\sup_m \lvert u_m\rvert_{W^{s,p}(G)} <\infty$. 
In the sequel, we shall restrict ourselves
to sufficiently large $m$ satisfying $L_m\subset\subset  R$. 
Fix  $v\in C^\infty_c(G')$ such that $v= 1$ on $R$. Fix 
$w_m\in C^\infty_c(L_m)$ satisfying $w_m= 1$ on $L_{2m}$  and
 $\lVert w_m\rVert_\infty + m^{-1}\lVert \nabla w_m\rVert_\infty \lesssim 1$. Now, the function $u_m =v-w_m\in C_c(G)$ satisfies
$u_m\ge 1$ on $R\setminus L_m$. Furthermore,
\begin{align*}
\lvert u_m\rvert_{W^{s,p}(G)}^p \lesssim \lvert v\rvert_{W^{s,p}(G)}^p + \lvert w_m\rvert_{W^{s,p}(G)}^p\,.
\end{align*}
Since $G$ is bounded and $\lvert v(x)-v(y)\rvert\lesssim \lvert x-y\rvert$ for every $x,y\in G$, we
find that $\lvert v\rvert_{W^{s,p}(G)}^p<\infty$. 

In order to estimate $\lvert w_m\rvert_{W^{s,p}(G)}^p$,  we let $E_m\subset L$ be a set
such that 
 $L \subset \cup_{z\in E_m} B(z,1/(2m))$ and
$\sharp E_m\lesssim m  = m^{2-sp}$. Now
\begin{align*}
\lvert w_m\rvert_{W^{s,p}(G)}^p  \le 2 \int_G\int_{L_m} \frac{\lvert w_m(x)-w_m(y)\rvert^p}{\lvert x-y\rvert^{n+sp}}\,dy\,dx =:2I_m\,.
\end{align*}
Note that $L_m\subset \cup_{z\in E_m} B(z,1/m)$. By writing $B_z:=B(z,1/m)$,
\begin{align*}
I_m &\le \sum_{z\in E_m} \int_G \int_{B_z}\frac{\lvert w_m(x)-w_m(y)\rvert^p}{\lvert x-y\rvert^{n+sp}}\,dy\,dx\\
&\le \sum_{z\in E_m} \sum_{l=0}^\infty \int_{B(z,(2l+2)/m)\setminus B(z,2l/m)}\int_{B_z}\frac{\lvert w_m(x)-w_m(y)\rvert^p}{\lvert x-y\rvert^{n+sp}}\,dy\,dx\,.
\end{align*}
For $l=0$ we use inequality $\lvert \nabla w_m\rvert \lesssim m$ and, for the remaining summands with $l=1,2,\ldots,$ we use
inequality $\lvert w_m(x)-w_m(y)\lvert \lesssim 1$ to conclude that
\[
I_m\lesssim m^{sp-2} \sum_{z\in E_m}\sum_{l=0}^\infty \bigg(\frac{1}{l+1}\bigg)^{sp+1} \lesssim 1\,.
\]
For further details on the first inequality above, we refer to \cite[\S 2]{Dyda}.
\end{proof}

\bibliographystyle{abbrv}
\def\cprime{$'$}

\end{document}